\documentclass[12pt,reqno]{article}
\usepackage{graphicx}
\usepackage{amsfonts}
\usepackage{amsmath, amsthm, amssymb}
\usepackage{hyperref}
\usepackage{cite} 
\newtheorem{theorem}{Theorem}[section]

\newtheorem{prop}{Proposition}[section]

\newcommand{\no}{\noindent}
{\rm} 
\begin{document}

\author{S. Molinaro, M. Hlynka and S. Xu\\  
Department of Mathematics and Statistics,\\
 University of Windsor, Windsor, Ontario, Canada N9B 3P4}
\date{August 3, 2011}
\title{Two Unordered Queues}

\maketitle
\no
\begin{abstract}
A  special customer must complete service from two servers in series, in either order, each with an $M/M/1$ queueing system. It is assumed that the two queueing system lengths are independent with initial numbers of customers $a$ and $b$ at the instant when the special customer arrives. We find the expected total time (ETT) for the special customer to complete service. We show that even if the interarrival and service time parameters of two queues are identical, there exist examples (specific values of the parameters and initial lengths) for which the special customer surprisingly has a lower expected total time to completion by joining the longer queue first rather than the shorter one. 
\end{abstract}

\no
{\bf e-mail: } molinaro@stececile.ca, hlynka@uwindsor.ca, xul@uwindsor.ca  
\newpage

\renewcommand{\baselinestretch}{2}
\section{\bf Introduction}

There is a great deal of literature on queues in series. One of the most famous articles is by Jackson (1957)\cite{jackson}. Issues of the ordering of the queues appear in works such as Weber (1979)\cite{weber}, Whitt (1985)\cite{whitt}, Anantharam (1987)\cite{anantharam}. Our model assumptions and goals differ from these articles. 

We begin with assumptions and the goal of this work. A special customer must receive service from two independent M/M/1 (Markovian) queueing systems with the same arrival and service rates $\lambda$ and $\mu$. 
At time 0, queueing system 1 has $a$ customers (including the customer in service) and queueing system 2 has $b$ customers. Given a choice, which queue should the special customer join first, in order to minimize its expected total system time (sum of two system times)?  Customers are served in a FCFS order for each server. No reneging (switching lines) is allowed. 
Letting $S1$ and $S2$ represent the servers, the special customer must choose between: 
\begin{align*}
\rightarrow  &* * * * S1      &\rightarrow  & ? ? ? ? ? ? ? S2\text{ or }\\
\rightarrow  &* * * * * * * S2&\rightarrow  & ? ? ? ?  S1
\end{align*}
Here ``*'' represents a customer at time 0 (and the numbers $a$ and $b$ are known to the special customer). ``?'' represents a customer in the second system after the first queueing system is completed. The number of customers after the first queueing system is completed is a random variable and so we use ``?'' to indicate that uncertainty.

The concept of a special customer as discussed here is not unreasonable.
Suppose I need to renew a document for my car and my friend also needs the same. My friend asks me to perform both renewals. Often the queues for such systems are set up by dividing the surnames by alphabet (say A-M and N-Z) in such a way that both groups have approximately the same arrival rate. The servers are equally capable so the service rates are also equal. My friend's surname is in one group and my surname is in the other group so I must join each line in tandem.  This describes a system with independent queues, with essentially the same arrival and service rates, and I am a special customer as nobody else is likely to have double tasks. 

Our intuition suggests that in order to have the lowest expected total system time for a pair of queues of the type described, the special customer should join the shorter queue first, and upon completion, the customer should then join the other queue (which may have its length changed during the time that the special customer is in the shorter queueing system). The reason for this intuition is that if the special customer joins the shorter queue first, then during the system time (including the customer in service and the special customer) for the short queue, the expected length of the long queue should drop, resulting in a lowering of total system time for the special customer. On the other hand, if  the special customer joins the long queue first, then during the system time for the long queue, the expected length of the short queue should increase, resulting in an increase in total system time for the special customer. This intuition is NOT completely correct and should be examined more closely. We consider three cases. \\
\bigskip

\noindent
Case 1: At time 0, queue 1 length is less than average and queue 2 is longer than average.\\ 
If queue 1 is selected first, then, during the special customer's time in system 1 (including the special customer and the customer in service),  the expected length of queue 2 should get shorter. \\ 
If queue 2 is selected first, then, during the special customer's system time in system 2, the expected length of queue 1 should increase. \\
Thus the shorter queue should be chosen first. 
\bigskip

\noindent 
Case 2: At time 0, queue 1 and queue 2 lengths are both less than average, with queue 1 less than queue 2.\\ 
If queue 1 is selected first, then, during the special customer's queue 1 system time, the expected length of queue 2 should get slightly longer (less time, lower rate). \\ 
If queue 2 is selected first, then, during the special customer's queue 2 system time, the expected length of queue 1 should get noticeably longer (more time, faster rate).\\
Thus the shorter queue should be chosen first (again).
\bigskip 

\noindent
Case 3: Both queue 1 and queue 2 are longer than than average, with queue 1 less than queue 2.\\ 
If queue 1 is selected first, then, during the special customer's queue 1 system time, the expected langth of queue 2 should get shorter (less time, higher rate). \\ 
If queue 2 is selected first, then, during the special customer's queue 2 system time, the expected length of queue 1 should get shorter (more time, lower rate).\\ 
In Case 3, because of the time and rates move in opposite directions, the choice of which queue to select first is not clear.  

Thus, with our more detailed analysis, our intuition suggests possible situations in which the special customer might achieve a shorter expected total time to completion by joining the longer queue first. To find such situations, we should seek values that belong to case 3 and where the change in the rates of drop of the queue length  are similar for both queues during the first queue's system time. In the next section, we give a more rigorous analysis for the expected total time to completion.  

\section{\bf Analysis}
We assume  two $M/M/1$ queueing systems acting independently of each other, reach with exponential interarrival times (arrival rate $\lambda$)  and exponential service times (rate $\mu$). We use the following notation. \\
$L_i(t)=$ queue length (including customer in service, at  time $t$, beginning with $i$ customers at time $0$.\\
$p_{ij}(t)=$ transient probability of an M/M/1 queueing system having $j$ customers at time $t$, given $i$ customers at time $0$.\\  
$EL_{i}(t)=$ the expected number of customers at time $t$ if there are $i$ customers at time 0. 
$T_n(t)=$ time to complete service on $n$ customers. \\
$f_n(t)=$ pdf of the time to service $n$ customers which is Erlang($n$). \\
\begin{equation*}f_{n}(t)=\dfrac{\mu^n t^{n-1}}{(n-1)!}e^{-\mu t}  \end{equation*}

\begin{theorem}A special customer must complete service in either order from two servers, each with its own queue. The two queues are independent M/M/1 with initial lengths $a$ and $b$ respectively and parameters $\lambda$ and $\mu$. The expected total time ($ETT_{ab}$)to complete both queues (with queue ``$a$'' first), is
\begin{equation*}
E(TT_{ab})=\dfrac{a+2}{\mu}+\int_{0}^{\infty}{f_{a+1}(t)\dfrac{1}{\mu}EL_b(t)dt}.
\end{equation*}
\end{theorem}  
 \begin{proof} 
\begin{align*}
E(TT_{ab})&=\dfrac{a+1}{\mu}+\int_{0}^{\infty}{f_{a+1}(t)\sum_{n=0}^{\infty}p_{b,n}(t)\dfrac{n+1}{\mu}dt}\\
&=\dfrac{a+2}{\mu}+\int_{0}^{\infty}{f_{a+1}(t)\dfrac{1}{\mu}EL_{b}(t)dt}.
\end{align*} 
\end{proof}

In order to compute $EL_{b}(t)$ in Theorem 2.1, we use two results presented in van de Coevering (1994)\cite{van}  for transient probabilities and expected values for an M/M/1 queue. 
\begin{theorem} For an M/M/1 queue, 
\begin{equation*}
p_{ij}(t)=\dfrac{2}{\pi}\rho^{(j-i)/2}\int_0^\pi \dfrac{e^{-\mu t \gamma(y)}}{\gamma(y)}a_i(y)a_j(y)\,dy
+ \begin{cases}(1-\rho)\rho^j & \rho<1\\
  0 & \rho \geq 1, \end{cases}
\end{equation*}
where 
\begin{equation*}\gamma(y)=1+\rho-2\sqrt{\rho}\cos(y) \text { and } a_k(y)=\sin(ky)-\sqrt{\rho}\sin ((k+1)y).
\end{equation*}
\end{theorem}

\begin{theorem} (van de Coevering) For an $M/M/1$ queueing system, let $E_{i}(t)$ be the expected number of customers at time $t$ if there are $i$ customers at time 0. Assume $\rho<1$. Then
\begin{equation*}
EL_{i}(t)=\dfrac{2}{\pi}\rho^{(1-i)/2}\int_0^\pi \dfrac{e^{-\mu t \gamma(y)}}{\gamma(y)^2}a_i(y)\sin(y)\,dy+ \rho/(1-\rho).
\end{equation*}\end{theorem}

Van de Coevering (1995)\cite{van} credits the statements of theorems 2.2 and 2.3 to work of Takacs (1962)\cite{takacs} and Morse (1955, 1958)\cite{morse1},\cite{morse2}.  These two theorems allow us to create a graph of $EL_{i}(t)$ for various initial values of $i$. Graphs of this type appear in Abate and Whitt (1988) and in Yu, He and Zhang (2006)\cite{yu}. 

\newpage

Consider Figure 1 which shows $EL_i(t)$ for $\lambda=3$, $\mu=4$, $0\leq t \leq10$, $i=0,1,\dots, 7$. This diagram is very interesting. All of the curves have the same negative slope at $t=0$. The expected number of customers in the system in steady state is given by $E(L)=\dfrac{\lambda}{\mu-\lambda}$, and equals 3 for 
$\lambda=3$, $\mu=4$. 

\begin{figure}
  \includegraphics[width=4in, height=3in]{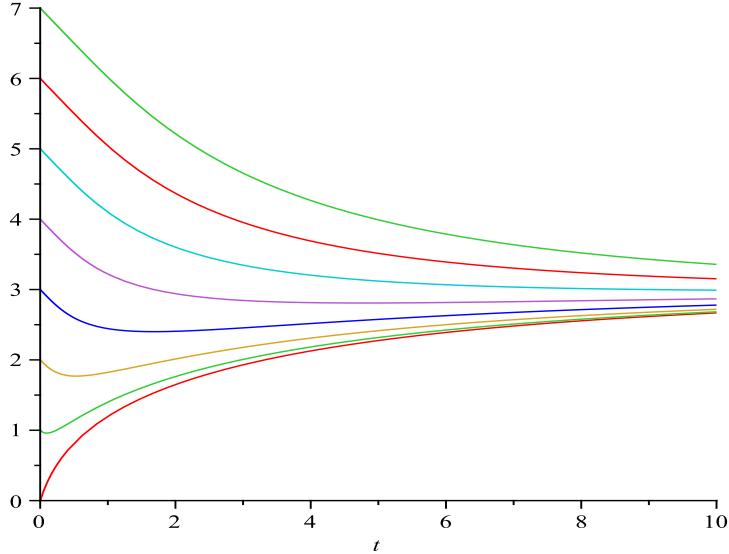}\\
\caption{$EL_{i}(t)$, $i=0,1,2,3,4,5,6,7$, $\lambda=3, \mu=4, 0<t<10$.}
\end{figure}

We observe that for lengths less than the steady state expected number, the curves all decrease before increasing to the steady state value.
Part of this observation  is given in the following result.     

\begin{prop}
For an M/M/1 queue with parameters $\lambda$ and $\mu$, 
\begin{equation*}
\left. \dfrac{dEL_{i}(t)}{dt}\right|_{t=0}=(\lambda-\mu) \text{ for }i=1,2,\dots . 
\end{equation*}
\end{prop}

Note that for large initial values $i$ at time $0$, the curve $EL_{i}(t)$ looks like a straight line with slope $\mu-\lambda$ before bending as $t$ gets larger.   
\bigskip

\begin{figure}
  \includegraphics[height=4in, width=5.0in]{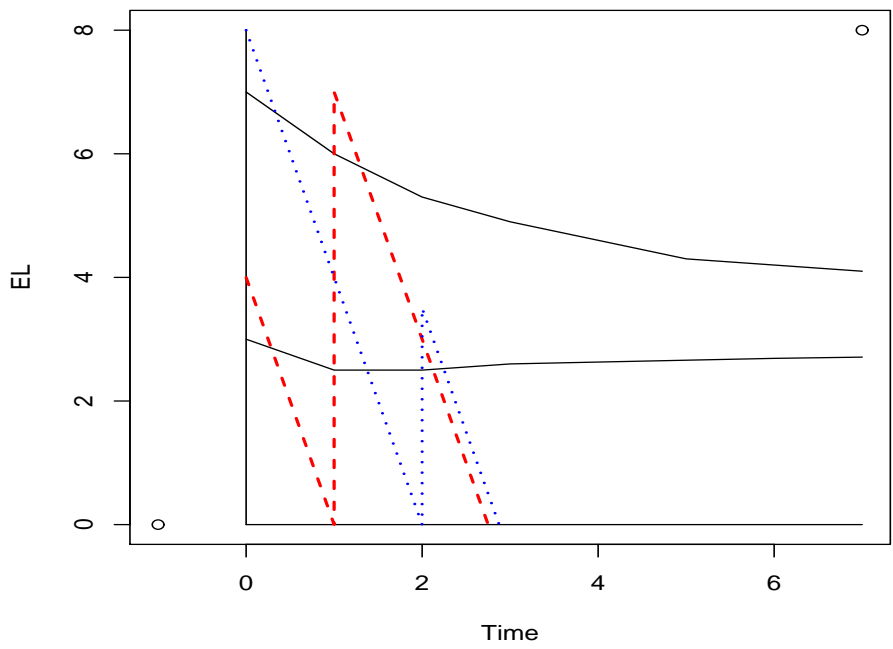}\\
\caption{$ETT_{i}$, $\lambda=3, \mu=4$, $Red (dashes):\, a=3,b=7$; 
 $Blue (dots):\,a=7,b=3$. }
\end{figure}

In Figure 2, the parameters are $\lambda=3$, $\mu=4$, with $a=3$, $b=7$, and the steady state expected length of the queueing system is $EL=\dfrac{\lambda}{\mu-\lambda}=3$. The slope of both curves $EL_{3}(t)$ and $EL_{7}(t)$ at $t=0$ is $\lambda-\mu=-1$. If the special customer chooses to join the queueing system with 3 customers, we can follow what happens on red (dashed) path in Diagram 2. When the special customer joins the 3 customers, there are then 4 customers to be served. The mean service rate is 4 customers so the mean time to complete service on the 4 customers (at rate $\mu=4$) is 1 time unit. At time 1 (approximately), the special customer joins the other queue. The other queue had 7 customers at time $t=0$, but by time $t=1$, the expected number has declined as shown on the graph to approximately 6. Adding the special customer to the 6 customers moves the red path up to approximately 7. The mean time to service these is approximately 7/4, giving a completion time of 1+7/4=2.75. The expression in Theorem 2.1 gives an expected Total Time which is slightly different from this value because it is averaged over all possible cases. 

In Figure 2, we also see what happens to the special customer if it joins the long queue first. We follow the blue (dotted) path in this case. When the special customer joins the 7 customers, there are then 8 customers to be served. The mean service rate is 4 customers so the mean time to complete service on the 8 customers (at rate $\mu=4$) is 2 time units. At time 2 (approximately), the special customer joins the other queue. The other queue had 3 customers at time $t=0$, but by time $t=1$, the expected number has changed but not by much since the curve $EL_{3}(t)$ curve drops and then increases to about 2.6. in the 2 time units.  Adding the special customer moves the blue path up to approximately 3.6. The mean time to service these is approximately 3.6/4, giving a completion time of 2+3.6/4=2.9. 

Thus in this example, the result is what we intuitively expect. Joining the shorter queue first leads to a lower expected total time to completion.      
\newpage

\begin{figure}
  \includegraphics[height=4.0in, width=5.0in]{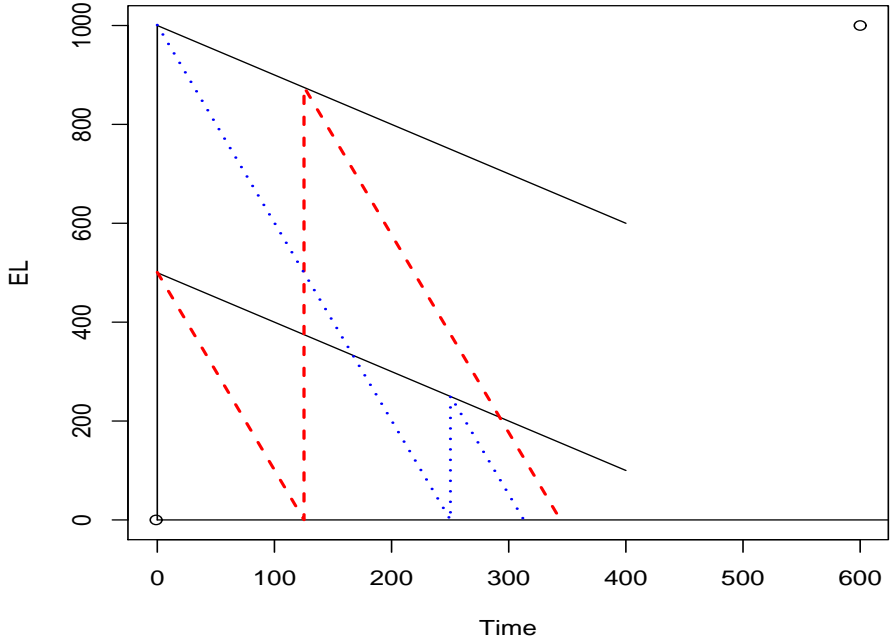}\\
\caption{$ETT_{i}$, $\lambda=3, \mu=4$, $Red (dashes):\, a=500, b=1000$; $Blue (dots):\, a=1000, b=500$.}
\end{figure}

Next we present an example for which it is advantageous to join the longer queue. 
In Figure 3, the parameters are $\lambda=3$, $\mu=4$,  with $a=500$, $b=1000$, and the steady state expected length of the queueing system is $EL=\dfrac{\lambda}{\mu-\lambda}=3$. The slope of both curves $EL_{500}(t)$ and $EL_{1000}(t)$ at $t=0$ is $\lambda-\mu=3-4=-1$, and this stays approximately true for a long time. If the special customer chooses to join the queueing system with 500 customers, we can follow what happens on red (dashed) path in Diagram 3. When the special customer joins the 500 customers, there are then 501 customers to be served. The mean service rate is 4 customers so the mean time to complete service on the 501 customers (at rate $\mu=4$) is 501/4 time units. At time 501/4 (approximately), the special customer joins the other queue. The other queue had 1000 customers at time $t=0$, but by time $t=501/4$, the expected number has declined as shown on the graph to approximately 1000-501/4=874.75. Adding the special customer to 874.75 moves the red path up to approximately 875.75. The mean time to service these is approximately 875.75/4, giving an expected total time to completion of approximately 501/4+875.75/4=344.1875. The expression in Theorem 2.1 would give an expected Total Time slightly different from this value because it is correctly averaged over all possible variables. 

In Figure 3, we also see what happens to the special customer if it joins the long queue first. We follow the blue (dotted) path in this case. When the special customer joins the 1000 customers, there are then 1001 customers to be served. The mean service rate is 4 customers per time unit so the mean time to complete service on the 1001 customers (at rate $\mu=4$) is $1001/4=250.25$ time units. At time 1001/4 (approximately), the special customer joins the other queue. The other queue had 500 customers at time $t=0$, but by time $t=250.25$, the expected number has changed to an average of $500-250.25=249.75$ customers. Adding the special customer moves the blue path up to approximately 250.75. The mean time to service these is approximately 250.75/4, giving an expected total time completion time of $1001/4+250.75/4= 312.9375$ time units. 

Thus in this example, the result is not what we intuitively expect. Joining the shorter queue first leads to a higher expected total time to completion.  

\section{Numerical Results}

The figures in the previous section convincingly show that sometimes the special customer should choose the shorter queue first and sometimes the customer should join the longer queue first. However, the correct 
version of the expected total time to completion is given in Theorem 2.1, whereas the diagrams are only approximate. 

We now present some results using Theorem 2.1 and comment on the results.  The exact expected value which corresponds to Figure 2 follows. \\ 
$\lambda =3$; $\mu=4$; $a=3$; $b=7$; $ETT=2.7607$\\
$\lambda =3$; $\mu=4$; $a=7$; $b=3$; $ETT=2.8560$ \\
We note that to minimize $ETT$, the special customer should join the shorter queue first and the expected value of the total time to completion is 2.7607. 

The result corresponding to Figure 3 is a fairly extreme case. The expected number of customers in the system is 3, yet we are working with $a=500$ and $b=1000$ customers. Such an unusual situation might occur if there is a one time major input into an otherwise pair of stable M/M/1 queues (perhaps an  unexpected busload of customers arrives from out-of-town, temporarily swamping the two queues). 

The reason for the success of the longer queue first strategy is that both $EL_{i}$ curves are decreasing in an approximately parallel way. 

In order to have a more reasonable set of values for $a$ and $b$, yet keeping the reason for the success of the longer queue first strategy , we take parameters  $\lambda =3$, $\mu=4$. We look at the two cases $a=10$, $b=12$ and $a=12$; $b=10$. The results follow:\\ 
$\lambda =3$; $\mu=4$; $a=10$; $b=12$; $ETT=5.3301$ \\
$\lambda =3$; $\mu=4$; $a=12$; $b=10$; $ETT=5.2599$\\

This shows that sometimes under reasonable circumstances, a customer should choose the longer queue first.  We can even make $a$ and $b$ lower and still have the special customer gain by joining the longer queue (although the gain is very small). \\
$\lambda =3$; $\mu=4$; $a=5$; $b=9$; $ETT=3.6354$\\
$\lambda =3$; $\mu=4$; $a=9$; $b=5$; $ETT=3.6254$\\

Both of the above examples fit into Case 3 of Section 2 (both queueing system lengths $a$ and $b$ are longer than the steady state average).  
We surmise that if the steady state expected length $EL$ is large, then the for $i$ near $EL$, the curve $EL_i$ has an intital decrease for a relatively long time before increasing again. Thus it might be possible to find examples belonging to Case 1 and Case 2 where it is also better to join the longer queue first. In order to achieve this, we seek parameters such that $EL$ is fairly large but $\mu-\lambda$ is small. 
We try $\lambda=9$, $\mu=10$. Then $EL=\dfrac{\lambda}{\mu-\lambda}=9$ and $\mu-\lambda=1$. We then choose $a=8,b=9$ and the reverse version $a=9,b=8$. This example is on the border of Case 1 and Case 2. The results are:\\ 
$\lambda =9$; $\mu=10$; $a=8$; $b=9$; $ETT=1.8181$\\
$\lambda =9$; $\mu=10$; $a=9$; $b=8$; $ETT=1.8177$

Again, the special customer has a lower ETT by choosing the longer queue first. If we change $\mu$ by $\pm\epsilon$ where $\epsilon>0$ is very small (say .00000001), this will not affect ETT (rounded to 4 decimal places). However, for $\epsilon >0$, $EL<9$ so the example belongs to Case 2; and for $\epsilon <0$, $EL>9$ so the example belongs to Case 1. thus we have found an example for which the special customer should join the longer queue first even though the shorter queue is shorter than the steady state average and the longer queue is longer than the steady state average! 

\section{Conclusions}
Contrary to intuitive reasoning, it is NOT ALWAYS true that a special customer should join the shorter queue when required to pass through two queues (in either order), even if all parameters (other than current length) are equal. Counterexamples (in extreme cases and fairly standard cases) have been presented. Even if one queueing system exceeds the steady state expected system length and the other is less than the expected length, we have constructed examples for which the special customer should join the longer queue first. 

As a warning, for large values of $a$ and $b$, the calculations can be time consuming on a computer. The calculations appear in the appendix. 

Future research on this subject could involve more than two queues.  Nonexponential distributions amy be added. One time jockeying (or multiple jockeying) could be allowed. It is possible that we will find that examples exist for which a special customer may save time by jockeying to a second queue if the second queue becomes longer! 

Snall amounts of the work appear in the technical report of Molinaro and Hlynka (2009)\cite{molinaro} and Xu (2010)\cite{xu}.

\newpage
\begin{center}
APPENDIX\\
\end{center}
\bigskip
MAPLE CODE\\
\bigskip
Notation:\\ 
g=gamma function as in Theorem 2.1\\
r=rho\\
m=mu\\
a=initial count in queue 1\\
b= initial count in queue 2\\
\bigskip

\noindent
$> a(k,y,r):=sin(k \cdot y)-\sqrt{r} \cdot sin((k+1) \cdot y)$:\\
$> g(y,r):=1+r-2\cdot \sqrt{r} \cdot cos(y):$\\
$>E(i,t,r,m):=\dfrac{2}{\pi}r^{\frac{(1-i)}{2}}\displaystyle\int_{0}^{\pi}\left.\right.\dfrac{e^{-m\cdot t\cdot g(y,r)}}{g(y,r)^2}\cdot a(i,y,r)\cdot sin(y)\,dy+\dfrac{r}{1-r}$:\\
$>ETT(a,b,r,m):=\dfrac{a+2}{m}+\displaystyle\int_{0}^{\infty}\dfrac{m^a}{a!}\cdot t^a \cdot e^{-m\cdot t}\cdot E(b,t,r,m)dt:$\\
$> evalf(ETT(3, 7, .75, 4));$

\qquad 2.760732021\\
$> evalf(ETT(7, 3, .75, 4));$

\qquad 2.856035058

\end{document}